\documentclass[11pt]{article}
\usepackage{amssymb}
\usepackage{amsfonts}
\usepackage{amsmath}

\newtheorem{theorem}{Theorem}[section]

\newtheorem{conjecture}[theorem]{Conjecture}

\newtheorem{lemma}[theorem]{Lemma}

\newtheorem{remark}[theorem]{Remark}

\newenvironment{proof}[1][Proof]{\noindent\textbf{#1.} }{\ \rule{0.5em}{0.5em}}
\input{tcilatex}
\begin{document}

\title{Euler characteristics and actions of automorphism groups of free
groups}
\author{Shengkui Ye \\
Xi'an Jiaotong-Liverpool University}
\maketitle

\begin{abstract}
Let $M^{r}$ be a connected orientable manifold with the Euler characteristic 
$\chi (M)\not\equiv 0\func{mod}6$. Denote by $\mathrm{SAut}(F_{n})$ the
unique subgroup of index two in the automorphism group of a free group. Then
any group action of $\mathrm{SAut}(F_{n})$ (and thus the special linear
group $\mathrm{SL}_{n}(\mathbb{Z})$) $(n\geq r+2$) on $M^{r}$ by
homeomorphisms is trivial. This confirms a conjecture related to Zimmer's
program for these manifolds.
\end{abstract}

\section{Introduction}

Let $\mathrm{SL}_{n}(\mathbb{Z})$ be the special linear group over integers.
There is an action of $\mathrm{SL}_{n}(\mathbb{Z})$ on the sphere $S^{n-1}$
induced by the linear action on the Euclidean space $\mathbb{R}^{n}.$ It is
believed that this action is minimal in the following sense.

\begin{conjecture}
\label{conj}Any action of $\mathrm{SL}_{n}(\mathbb{Z})$ $(n\geq 3)$ on a
compact connected $r$-manifold by homeomorphisms factors through a finite
group action if $r<n-1.$
\end{conjecture}

The smooth version of this conjecture was formulated by Farb and Shalen \cite%
{fs}, and is related to the Zimmer program concerning group actions of
lattices in Lie groups on manifolds (see the survey articles \cite{fi,zm}
for more details). When $r=1,$ Conjecture 1.1 is already proved by Witte 
\cite{wi}. Weinberger \cite{we} confirms the conjecture when $M=T^{r}$ is a
torus. Bridson and Vogtmann \cite{bv} confirm the conjecture when $M=S^{r}$
is a sphere. Ye \cite{ye1} confirms the conjecture for all flat manifolds.
For $C^{1+\beta }$ group actions of finite-index subgroup in $\mathrm{SL}%
_{n}(\mathbb{Z}),$ one of the results proved by Brown, Rodriguez-Hertz and
Wang \cite{brw} confirms Conjecture \ref{conj} for surfaces. For $C^{2}$
group actions of cocompact lattices, Brown-Fisher-Hurtado \cite{bfh}
confirms Conjecture \ref{conj}. Note that the $C^{0}$ actions could be very
different from smooth actions. It seems that very few other cases have been
confirmed (for group actions preserving extra structures, many results have
been obtained, cf. \cite{fi,zm}).

Let $\mathrm{SAut}(F_{n})$ denote the unique subgroup of index two in the
automorphism group $\mathrm{Aut}(F_{n})$ of the free group $F_{n}$. Note
that there is a surjection $\phi :\mathrm{SAut}(F_{n})\rightarrow \mathrm{SL}%
_{n}(\mathbb{Z})$ given by the abelianization of $F_{n}$. In this note, we
obtain the following general result on topological actions.

\begin{theorem}
\label{main}Let $M^{r}$ be a connected (resp. orientable) manifold with the
Euler characteristic $\chi (M)\not\equiv 0\func{mod}3$ (resp. $\chi
(M)\not\equiv 0\func{mod}6$). Then any group action of $\mathrm{SAut}(F_{n})$
$(n>r+1$) on $M^{r}$ by homeomorphisms is trivial.
\end{theorem}

Since any group action of $\mathrm{SL}_{n}(\mathbb{Z})$ could be lifted to
an action of $\mathrm{SAut}(F_{n}),$ Theorem \ref{main} confirms Conjecture %
\ref{conj} for orientable manifolds with nonvanishing Euler characteristic
modulo $6$.

\begin{remark}
\textit{(i) The bound of }$n$\textit{\ cannot be improved, since }$\mathrm{%
SAut}(F_{n})$\textit{\ acts through }$\mathrm{SL}_{n}(\mathbb{Z})$ \textit{%
non-trivially on }$S^{n-1}.$

\textit{(ii) Belolipetsky and Lubotzky \cite{bl} prove that for any finite
group }$G$\textit{\ and any dimension }$r\geq 2$\textit{, there exists a
hyperbolic manifold }$M^{r}$\textit{\ such that }$\mathrm{Isom}(M)\cong G.$%
\textit{\ Therefore, }$\mathrm{SL}_{n}(\mathbb{Z})$ and thus\textit{\ }$%
\mathrm{SAut}(F_{n})$ \textit{could act nontrivially through a finite
quotient group on such a hyperbolic manifold. This implies that the
condition of Euler characteristic could not be dropped.}

(iii) To satisfy the assumption on the Euler characteristic, the dimension $%
r $ has to be even. There are however no further restrictions on $r$, as the
following example shows. Let $\{g_{i}\}$ be a sequence of nonnegative
integers with $g_{i}\not\equiv 1\func{mod}3$ and $\Sigma _{g_{i}}$ an
orientable surface of genus $g_{i}.$ For any even number $r,$ 
\begin{equation*}
M^{r}=\Sigma _{g_{1}}\times \Sigma _{g_{2}}\times \cdots \times \Sigma _{g_{%
\frac{r}{2}}}
\end{equation*}%
has nonzero (modulo $6$) Euler characteristic and thus satisfies the
condition of Theorem \ref{main}.
\end{remark}

Our proof of Theorem \ref{main} rely on torsion and so will not be
applicable to finite-index subgroups$.$ Actually, Theorem \ref{main} does
not hold for general finite-index subgroups. For example, let $q<\mathbb{Z}$
be a non-trivial ideal and $C$ a non-trivial cyclic subgroup of $\mathrm{SL}%
_{n}(\mathbb{Z}/q).$ Let $f:\mathrm{SAut}(F_{n})\overset{\phi }{\rightarrow }%
\mathrm{SL}_{n}(\mathbb{Z})\rightarrow \mathrm{SL}_{n}(\mathbb{Z}/q)$ be the
group homomorphism induced by quotient ring homomorphism. The group $%
f^{-1}(C)$ could act non-trivially on $S^{2}$ by rotations through $C.$ By a
profound result of Grunewald and Lubotzky \cite{gl} (Corollary 1.2), there
is a group homomorphism $\rho $ from a finite-index subgroup $G$ of $\mathrm{%
SAut}(F_{n})$ $(n\geq 3)$ to $\mathrm{SL}_{n-1}(\mathbb{Z})$ such that $%
\func{Im}f$ is of finite index. Therefore, the group $G$ could act through $%
\mathrm{SL}_{n-1}(\mathbb{Z})$ on $S^{n-2}$, which is an infinite-group
action.

\section{Proofs}

The cohomology $n$-manifold mod $p$ (a prime) considered in this article
will be as in Borel \cite{Bo}. Roughly speaking, a cohomology $n$-manifold
mod $p$ is a locally compact Hausdorff space which has a local cohomology
structure (with coefficient group $\mathbb{Z}/p$) resembling that of
Euclidean $n$-space. Let $L=\mathbb{Z}$ or $\mathbb{Z}/p.$ All homology
groups in this section are Borel-Moore homology groups with compact supports
and coefficients in a sheaf $\mathcal{A}$ of modules over $L$. The homology
groups of $X$ are denoted by $H_{\ast }^{c}(X;\mathcal{A})$ and the
Alexander-Spanier cohomology groups (with coefficients in $L$ and compact
supports) are denoted by $H_{c}^{\ast }(X;L).$ We define the cohomology
dimension $\dim _{L}X=\min \{n\mid H_{c}^{n+1}(U;L)=0$ for all open $%
U\subset X\}.$ If $L=\mathbb{Z}/p,$ we write $\dim _{p}X$ for $\dim _{L}X.$
For integer $k\geq 0,$ let $\mathcal{O}_{k}$ denote the sheaf associated to
the pre-sheaf $U\longmapsto H_{k}^{c}(X,X\backslash U;L).$ An $n$%
-dimensional homology manifold over $L$ (denoted $n$-hm$_{L}$) is a locally
compact Hausdorff space $X$ with $\dim _{L}X<+\infty $, and $\mathcal{O}%
_{k}(X;L)=0$ for $p\neq n$ and $\mathcal{O}_{n}(X;L)$ is locally constant
with stalks isomorphic to $L$. The sheaf $\mathcal{O}_{n}$ is called the
orientation sheaf. There is a similar notion of cohomology manifold over $L$%
, denoted $n$-cm$_{L}$ (cf. \cite{b}, p.373). Topological manifolds are
(co)homology manifolds over $L.$

In order to prove Theorem \ref{main}, we need several lemmas.

\begin{lemma}[Borel \protect\cite{Bo}, Theorem 4.3, p.182 ]
\label{borel}Let $G$ be an elementary $p$-group operating on a first
countable cohomology $n$-manifold $X$ mod $p.$ Let $x\in X$ be a fixed point
of $G$ on $X$ and let $n(H)$ be the cohomology dimension mod $p$ of the
component of $x$ in the fixed point set of a subgroup $H$ of $G.$ If $%
r=n(G), $ we have 
\begin{equation*}
n-r=\sum_{H}(n(H)-r)
\end{equation*}%
where $H\ $runs through the subgroups of $G$ of index $p.$
\end{lemma}

\begin{lemma}[Mann and Su \protect\cite{ms}, Theorem 2.2]
\label{lemm1}Let $G$ be an elementary $p$-group of rank $k$ operating
effectively on a first countable connected cohomology $r$-manifold $X$ mod $%
p $. Suppose $\mathrm{dim}_{p}F(G)=r_{0}\geq 0$ where $F(G)$ is the fixed
point set of $G$ on $X$. Then $k\leq \frac{r-r_{0}}{2}$ if $p\neq 2$ and $%
k\leq r-r_{0}$ if $p=2.$
\end{lemma}

Let $X$ be an oriented cohomology $r$-manifold $X$ over $\mathbb{Z}$ (in the
sense of Bredon \cite{bre2}). A homeomorphism $f:X\rightarrow X$ is
orientation-preserving if the orientation is preserved. In the following
lemma, we consider the case of orientation-preserving actions.

\begin{lemma}
\label{lem3}Let $G$ be a non-trivial elementary $2$-group of rank $k$
operating effectively on a first countable connected oriented cohomology $r$%
-manifold $X$ over $\mathbb{Z}$ by orientation-preserving homeomorphisms.
Suppose $\mathrm{dim}_{2}F(G)=r_{0}\geq 0$ where $F(G)$ is the fixed point
set of $G$ on $X$. Then $k\leq r-1-r_{0}.$
\end{lemma}

\begin{proof}
Note that the manifold $X$ is also a cohomology $r$-manifold over $\mathbb{Z}%
/2$ and the fixed point set $\mathrm{Fix}(g)$ is a cohomology manifold over $%
\mathbb{Z}/2$ by Smith theory (see \cite{Bo}, Theorem 2.2 and the bottom of
p.78). If there is a non-trivial element $g\in G$ such that the dimension of
the fixed point set $\mathrm{Fix}(g)\ $is $r,$ the element $g$ acts
trivially by invariance of domain (see Bredon \cite{b}, Cor. 16.19, p.383).
This is a contradiction to the assumption that $G$ acts effectively.
Therefore, we could assume that $\mathrm{Fix}(g)$ is of non-trivial even
codimension by Bredon \cite{Bo} (Theorem 2.5, p.79. We use the assumption
that $X$ is a cohomology manifold over $\mathbb{Z}$). Now the lemma becomes
obvious for $r=1.$ When $r=2,$ the dimension of $\mathrm{Fix}(H)$ is zero
for any nontrivial subgroup $H<G.$ If $r=2$ and $k=1,$ the fixed point set $%
\mathrm{Fix}(G)$ is of dimension $0$ and the statement holds. If $r=2$ and $%
k\geq 2,$ this would be impossible by Borel's formula in Lemma \ref{borel}.

Choose a nontrivial element $g\in G$ such that the fixed point set $\mathrm{%
Fix}(g)\ $is of the maximal dimension among all nontrivial elements in $G$.
Fix a connected component $M$ of $\mathrm{Fix}(g)$ containing a connected
component of $F(G)$ with the largest dimension. Choose a decompostion $%
G=\langle g\rangle \bigoplus G_{0}$ for some subgroup $G_{0}<G.$ The action
of the complement $G_{0}$ leaves $M$ invariant. If some nontrivial element $%
h\in G_{0}$ acts trivially on $M,$ let $H=\langle g,h\rangle \cong (\mathbb{Z%
}/2)^{2}.$ By the assumption that the fixed point set $\mathrm{Fix}(g)\ $is
of the maximal dimension, each nontrivial element in $H$ has its fixed point
set of dimension $\dim _{2}\mathrm{Fix}(g).$ This is impossible by Borel's
formula in Lemma \ref{borel}. Therefore, the action of $G_{0}$ on $M$ is
effective. Note that $\mathrm{Fix}(g)$ is a cohomology manifold over $%
\mathbb{Z}/2$ (by Smith theory) of dimension at most $r-2.$ Thus the rank of 
$G_{0}$ is at most $r-2-r_{0}$ by Lemma \ref{lemm1}. Therefore,%
\begin{equation*}
k=\mathrm{rank}(G_{0})+1\leq r-1-r_{0}.
\end{equation*}
\end{proof}

The inequality in Lemma \ref{lem3} is sharp, by considering the linear
action of the diagonal subgroup $(\mathbb{Z}/2)^{n-1}<\mathrm{SL}_{n}(%
\mathbb{Z})$ on $\mathbb{R}^{n}.$

Let $X$ be a locally compact Hausdorff space and a finite group $G=(\mathbb{Z%
}/p)^{n}$ acting on $X$ by homeomorphisms. In the remaining part of this
article, we suppose that the Euler characteristic $\sum_{i}(-1)^{i}\dim
H^{i}(X;\mathbb{Z}/p)=:\chi (X;\mathbb{Z}/p)$ is defined. The following
results are well-known from Smith theory (cf. \cite{Bo}, Theorem 3.2 on page
40 and Theorem 4.4 on page 42-43).

\begin{lemma}
\label{lem4}We have the following.

\begin{enumerate}
\item[(i)] Suppose that the cyclic group $G=\mathbb{Z}/p$ operates freely on 
$X,$ whose $\dim _{\mathbb{Z}}X<\infty $ and $H^{\ast }(X;\mathbb{Z}/p)$ is
finite dimensional. Then $H^{\ast }(X/G;\mathbb{Z}/p)$ is finite-dimensional
and 
\begin{equation*}
\chi (X;\mathbb{Z}/p)=p\chi (X/G;\mathbb{Z}/p).
\end{equation*}

\item[(ii)] Suppose that the cyclic group $\mathbb{Z}/p$ operates on $X,$
whose $\dim _{\mathbb{Z}/p}X<\infty $ and $H^{\ast }(X;\mathbb{Z}/p)$ is
finite dimensional. Let $F$ be the fixed point set$.$ Then $H^{\ast }(F;%
\mathbb{Z}/p),H^{\ast }(X-F;\mathbb{Z}/p)$ are finite dimensional and 
\begin{equation*}
\chi (X;\mathbb{Z}/p)=\chi (X-F;\mathbb{Z}/p)+\chi (F;\mathbb{Z}/p).
\end{equation*}
\end{enumerate}
\end{lemma}

Denote by $G_{x}$ the stabilizer of $x\in X.$ Suppose that $X=\cup
_{i=0}^{n}X_{i}$ is the union of subspaces $X_{i}=\{x\in X:\mathrm{order}%
(G_{x})=p^{i}\}.$ It is clear that each $X_{i}$ is $G$-invariant and $X_{n}=%
\mathrm{Fix}(G)$.

\begin{theorem}
\label{preprop}\bigskip Suppose that $G\ $is a (not necessarily abelian) $p$%
-group of order $p^{n}$ acting on $X.$ Then 
\begin{equation*}
\chi (X;\mathbb{Z}/p)=\tsum\nolimits_{i=0}^{n}\chi (X_{i};\mathbb{Z}%
/p)=\tsum\nolimits_{i=0}^{n}p^{n-i}a_{i},
\end{equation*}%
for some integers $a_{i}.$ Actually, we have $\chi (X_{i};\mathbb{Z}%
/p)=p^{n-i}a_{i}.$
\end{theorem}

\begin{proof}
We prove the theorem by induction on $n.$ When $n=0,$ the statement is
trivial by the assumption that the Euler characteristic $\chi (X;\mathbb{Z}%
/p)$ is defined. When $n=1,$ this is Lemma \ref{lem4} by noting that $%
F=X_{1} $ and $X_{0}=X-F$. Choose $a$ to be an order-$p$ element in the
center of $G. $ Let $F=\mathrm{Fix}(a)$ and $X_{0}=X-F$. The quotient group $%
G/\langle a\rangle $ acts on the quotient space $X_{0}/\langle a\rangle $
and $F.$ Denote by 
\begin{equation*}
Y_{i}=\{x\in (X-F)/\langle a\rangle :|(G/\langle a\rangle )_{x}|=p^{i}\}
\end{equation*}%
and 
\begin{equation*}
Z_{i}=\{x\in F:|(G/\langle a\rangle )_{x}|=p^{i}\}.
\end{equation*}%
We will denote $\chi (X;\mathbb{Z}_{p})$ by $\chi (X)$ for short. By the
induction step, we have that 
\begin{equation*}
\chi ((X-F)/\langle a\rangle )=\sum_{i=0}^{n-1}\chi
(Y_{i})=\sum_{i=0}^{n-1}p^{n-1-i}a_{i}^{\prime }
\end{equation*}%
and 
\begin{equation*}
\chi (F)=\sum_{i=0}^{n-1}\chi (Z_{i})=\sum_{i=0}^{n-1}p^{n-1-i}b_{i}.
\end{equation*}%
The first equality in the statement of the theorem is proved by noting that $%
X_{i}=q^{-1}(Y_{i})\cup Z_{i-1}$ with the convention that $Z_{-1}=\emptyset
, $ where $q:(X-F)\rightarrow (X-F)/\langle a\rangle $ is the projection.
Therefore, we have 
\begin{eqnarray*}
\chi (X) &=&\chi (X-F)+\chi (F) \\
&=&p\chi ((X-F)/\langle a\rangle )+\chi (F) \\
&=&p^{n}a_{0}^{\prime }+\sum_{i=1}^{n-1}p^{n-i}(a_{i}^{\prime
}+b_{i-1})+b_{n-1}.
\end{eqnarray*}%
The proof is finished by choosing $a_{0}=a_{0}^{\prime },a_{i}=$ $%
a_{i}^{\prime }+b_{i-1}$ for $1\leq i\leq n-1$ and $a_{n}=b_{n-1}.$ The last
statement that $\chi (X_{i};\mathbb{Z}/p)=p^{n-i}a_{i}$ could be proved by
noting $X_{i}=q^{-1}(Y_{i})\cup Z_{i-1}$ and a similar induction argument.
\end{proof}

\bigskip

For a group $G$ and a prime $p,$ let the $p$-rank be $\mathrm{rk}%
_{p}(G)=\sup \{k\mid (\mathbb{Z}/p)^{k}\hookrightarrow G\}$. It is possible
that $\mathrm{rk}_{p}(G)=+\infty .$

\begin{theorem}
\label{prop}Let $M^{r}$ be a first countable connected cohomology $r$%
-manifold over $\mathbb{Z}/p$ and $\mathrm{Homeo}(M)$ the group of
self-homeomorphisms. We adapt the convention that $p^{n}=1$ when $n<0.$ Then
the $p$-rank satisfies%
\begin{equation*}
p^{\mathrm{rk}_{p}(\mathrm{Homeo}(M))-[\frac{r}{2}]}\mid \chi (M;\mathbb{Z}%
/p)
\end{equation*}%
when $p$ is odd and 
\begin{equation*}
2^{\mathrm{rk}_{2}(\mathrm{Homeo}(M))-r}\mid \chi (M;\mathbb{Z}/2)
\end{equation*}%
when $p=2.$ If $M^{r}$ $(r\geq 1)$ is an oriented connected cohomology $r$%
-manifold over $\mathbb{Z}$ and $\mathrm{Homeo}_{+}(M)$ is the group of
orientation-preserving self-homeomorphisms, we have 
\begin{equation*}
2^{\mathrm{rk}_{2}(\mathrm{Homeo}(M))-r+1}\mid \chi (M;\mathbb{Z}/2).
\end{equation*}
\end{theorem}

\begin{proof}
Suppose that an elementary $p$-group $G=(\mathbb{Z}/p)^{n_{p}}$ acts
effectively on $M$ for $n_{p}=\mathrm{rk}_{p}(\mathrm{Homeo}(M)).$ If the
group action is free, we have $p^{n_{p}}\mid \chi (M)$ by Theorem \ref%
{preprop} and the statements are obvious. In the following, we suppose that
the group action is not free. We let $X=M$ and $X_{i}$ as in Theorem 2.5 for
the sake of sticking to the notation of Theorem 2.5. Denote by $G_{x}$ the
stabilizer of $x\in X_{i}$ for nonempty $X_{i}.$ By Lemma \ref{lemm1} we
have $i:=\mathrm{rank}(G_{x})\leq \frac{r}{2}$ if $p\neq 2$ and $\mathrm{rank%
}(G_{x})\leq r$ if $p=2.$ Therefore, we have $p^{n_{p}-i}\geq p^{n_{p}-\frac{%
r}{2}}$ when $p\neq 2$ (or $p^{n_{p}-i}\geq p^{n_{p}-r}$ when $p=2$). This
implies that $p^{n_{p}-[\frac{r}{2}]}\mid \chi (M;\mathbb{Z}/p)$ (or $%
2^{n_{2}-r}\mid \chi (M;\mathbb{Z}/2)$ when $p=2$) considering Theorem \ref%
{preprop}. A similar argument proves the orientation-preserving case using
Lemma \ref{lem3}, by noting that the subgroup $G_{x}$ acts on $M$
orientation-preservingly if so does $G$.
\end{proof}

\begin{remark}
When $M$ is a surface, Theorem \ref{prop} was already known to Kulkarni \cite%
{fu}.
\end{remark}

Fixing a basis $\{a_{1},\ldots ,a_{n}\}$ for the free group $F_{n},$ we
define several elements in $\mathrm{Aut}(F_{n})$ as the following. The
inversions are defined as 
\begin{equation*}
e_{i}:a_{i}\longmapsto a_{i}^{-1},a_{j}\longmapsto a_{j}\text{ }(j\neq i);
\end{equation*}%
while the permutations are 
\begin{equation*}
(ij):a_{i}\longmapsto a_{j},a_{j}\longmapsto a_{i},a_{k}\longmapsto a_{k}%
\text{ }(k\neq i,j).
\end{equation*}%
The subgroup $N<\mathrm{Aut}(F_{n})$ generated by all $e_{i}$ $(i=1,\ldots
,n)$ is isomorphic to $(\mathbb{Z}/2)^{n}.$ The subgroup $W_{n}<\mathrm{Aut}%
(F_{n})$ is generated by $N$ and all $(ij)$ $(1\leq i\neq j\leq n).$ Denote $%
SW_{n}=W_{n}\cap \mathrm{SAut}(F_{n})$ and $SN=N\cap \mathrm{SAut}(F_{n}).$
The element $\Delta =e_{1}e_{2}\cdots e_{n}$ is central in $W_{n}$ and lies
in $\mathrm{SAut}(F_{n})$ precisely when $n$ is even.

The following result is Proposition 3.1 of \cite{bv}.

\begin{lemma}
\label{lemlast}Suppose $n\geq 3$ and let $f$ be a homomorphism from $\mathrm{%
SAut}(F_{n})$ to a group $G$. If $f|_{SW_{n}}$ has non-trivial kernel $K$,
then one of the following holds:

1. $n$ is even, $K=\langle \Delta \rangle $ and $f$ factors through $\mathrm{%
PSL}(n,\mathbb{Z})$,

2. $K=SN$ and the image of $f$ is isomorphic to $\mathrm{SL}(n,\mathbb{Z}/2)$%
, or

3. $f$ is the trivial map.
\end{lemma}

When $n=2m$ is even, for each $1\leq i\leq m$ define $R_{i}:F_{n}\rightarrow
F_{n}$ as $a_{2i-1}\longmapsto a_{2i}^{-1},a_{2i}\longmapsto
a_{2i}^{-1}a_{2i-1},a_{j}\longmapsto a_{j}$ $(j\neq 2i,2i-1).$ Let $T<%
\mathrm{SAut}(F_{n})$ be the subgroup generated by all $R_{i},$ $i=1,\ldots
,m.$ By Lemma 3.2 of Bridson-Vogtmann \cite{bv}, $T$ is isomorphic to $(%
\mathbb{Z}/3)^{m}.$ The following result is Proposition 3.4 of \cite{bv}.

\begin{lemma}
\label{lem10}For $m\geq 2$ and any group $G$, let $\phi :\mathrm{SAut}%
(F_{2m})\rightarrow G$ be a homomorphism. If $\phi |_{T}$ is not injective,
then $\phi $ is trivial.
\end{lemma}

\begin{proof}[Proof of Theorem \protect\ref{main}]
Let $f:\mathrm{SAut}(F_{n})\rightarrow \mathrm{Homeo}(M)$ be a group
homomorphism. Since the Euler characteristics $\chi (M;\mathbb{Z}/2)=\chi (M;%
\mathbb{Z}/3)$ (cf. \cite{bre2}, Theorem 5.2 and Corollary 5.7), they will
be simply denoted by $\chi (M)$. Since any action of $\mathrm{SAut}(F_{n})$
on a non-orientable manifold $M$ can be uniquely lifted to be an action on
the orientable double covering $\bar{M}$ (cf. \cite{b}, Cor. 9.4, p.67), we
may assume that $M$ is oriented and the group action is
orientation-preserving by noting that $\mathrm{SAut}(F_{n})$ is perfect (cf. 
\cite{ger}). When $M$ is non-orientable and $\chi (M)\not\equiv 0\func{mod}%
3, $ we would still have $\chi (\bar{M})\not\equiv 0\func{mod}3.$

When $n=3,$ the manifold $M$ is of dimension one. This case is already
proved by Bridson and Vogtmann \cite{bv}. Suppose that $n\geq 4.$ Choose $m=[%
\frac{n}{2}]$ (the integer part) and $T\cong (\mathbb{Z}/3)^{m}$. Let $%
\mathrm{SAut}(F_{2m})$ be the subgroup of $\mathrm{SAut}(F_{n})$ fixing $%
a_{n}$ if $n$ is odd. Note that $\mathrm{SAut}(F_{n})$ is normally generated
by a Nielsen automorphism in $\mathrm{SAut}(F_{2m})$ (cf. \cite{ger}). If $f$
is not trivial, the restriction $f|_{\mathrm{SAut}(F_{2m})}$ is not trivial
and thus the map $f|_{T}$ is injective by Lemma \ref{lem10}. Theorem \ref%
{prop} implies that $3\mid \chi (M),$ by noting that $n-r\geq 2.$ This is a
contradiction in the non-orientable case. If $\func{Im}f\ $contains a copy
of $(\mathbb{Z}/2)^{n-2}$, Theorem \ref{prop} would imply that $2\mid \chi
(M).$ This would be a contradiction to the assumption that $\chi
(M)\not\equiv 0\func{mod}6$ for the orientable manifold $M.$ Therefore, the
restriction $f|_{SN}$ is not injective and Case 1 in Lemma \ref{lemlast}
cannot happen, since $SN\cong (\mathbb{Z}/2)^{n-1}$. If case 2 happens, the
image satisfies $\func{Im}f=\mathrm{SL}(n,\mathbb{Z}/2).$ Let $x_{1i}(1)$
denote the matrix with 1s along the diagonal, 1 in the $(1,i)$-th position
and zeros elsewhere. Since the subgroup $\langle x_{12}(1),x_{13}(1),\cdots
,x_{1n}(1)\rangle \cong (\mathbb{Z}/2)^{n-1},$ we still have $2\mid \chi
(M). $ This is a contradiction, which implies that $f$ has to be trivial.
\end{proof}

\bigskip

From the above proof, we see that Theorem \ref{main} also holds for
cohomology manifolds over $\mathbb{Z}$.

\begin{remark}
For a specific $n,$ the conditions of Theorem \ref{main} may be improved.
For example, when $n$ is odd, Case 1 in Lemma \ref{lemlast} cannot happen. A
similar proof as that of Theorem \ref{main} shows that any action of $%
\mathrm{SAut}(F_{2k+1})$ $(k\geq 1)$ on an orientable manifold $M^{r}$ with $%
\chi (M)\not\equiv 0\func{mod}12$ (resp. $\chi (M)\not\equiv 0\func{mod}2$)
by homeomorphisms is trivial when $2k>r$ (resp. $2k\geq r$).
\end{remark}

\noindent \textbf{Acknowledgements}

The author would like to thank the referee for his/her detailed comments on
a previous version of this article and suggestions on the improvement. The
author is grateful to Prof. Xuezhi Zhao at Capital Normal University for
many helpful discussions. This work is supported by NSFC No. 11501459 and
Jiangsu Science and Technology Programme BK20140402.

Department of Mathematical Sciences, Xi'an Jiaotong-Liverpool University,
111 Ren Ai Road, Suzhou, Jiangsu 215123, China.

E-mail: Shengkui.Ye@xjtlu.edu.cn


\bigskip

\begin{thebibliography}{99}
\bibitem{bl} M. Belolipetsky and A. Lubotzky, \textit{Finite groups and
hyperbolic manifolds}, Invent. Math. \textbf{162}(2005), 459-472.

\bibitem{Bo} A. Borel, \textit{Seminar on transformation group,} Annals of
Mathematics Studies, No. 46, Princeton University Press, Princeton, N.J.
1960.

\bibitem{b} G. Bredon, \textit{Sheaf Theory}, second edition. Graduate Texts
in Mathematics, 170. Springer-Verlag, New York, 1997.

\bibitem{bre2} G. Bredon, \textit{Orientation in generalized manifolds and
applications to the theory of transformation group,} Michigan Math. Journal 
\textbf{7}(1960), 35-64.

\bibitem{bv} M. Bridson and K. Vogtmann, \textit{Actions of automorphism
groups of free groups on homology spheres and acyclic manifolds,}
Commentarii Mathematici Helvetici \textbf{86}(2011), 73-90.

\bibitem{bfh} A. Brown, D. Fisher and S. Hurtado, \textit{Zimmer's
conjecture: Subexponential growth, measure rigidity, and strong property (T)}%
, arXiv:1608.04995.

\bibitem{brw} A. Brown, F. Rodriguez Hertz and Z. Wang, \textit{Invariant
measures and measurable projective factors for actions of higher-rank
lattices on manifolds}, arXiv:1609.05565.

\bibitem{fs} B. Farb and P. Shalen, \textit{Real-analytic action of lattices}%
, Invent. math. \textbf{135}(1999), 273-296.

\bibitem{fi} D. Fisher, \textit{Groups acting on manifolds: Around the
Zimmer program}, In Geometry, Rigidity, and Group Actions 72-157. Univ.
Chicago Press, Chicago, 2011.

\bibitem{ger} S.M. Gersten, A presentation for the special automorphism
group of a free group, J. Pure Appl. Algebra, 33 (1984).

\bibitem{gl} F. Grunewald and A. Lubotzky, \textit{Linear representations of
the automorphism group of a free group}, Geom. Funct. Anal. 18 (2009), no.
5, 1564-1608.

\bibitem{fu} R. Kulkarni, \textit{Symmetries of surfaces}, Topology, \textbf{%
26}(1987), 195-203.

\bibitem{ms} L.N. Mann and J.C. Su, \textit{Actions of elementary p-groups
on manifolds}, Trans. Amer. Math. Soc., 106 (1963), 115-126.

\bibitem{mi} J. Milnor, \textit{Introduction to algebraic K-theory},
Princeton University Press, Princeton, N.J., 1971. Annals of Mathematics
Studies, No. 72.

\bibitem{wi} D. Witte, \textit{Arithmetic groups of higher }$Q$\textit{-rank
cannot act on }$1$\textit{-manifolds}, Proc. AMS, \textbf{122}(1994),
333-340.

\bibitem{we} S. Weinberger, $\mathrm{SL}(n,\mathbb{Z})$\textit{\ cannot act
on small tori}, \textit{Geometric topology} (Athens, GA, 1993), 406-408,
AMS/IP Stud. Adv. Math., Amer. Math. Soc., Providence, RI, 1997.

\bibitem{ye} S. Ye, \textit{Low-dimensional representations of matrix groups
and group actions on CAT(0) spaces and manifolds}, J. Algebra \textbf{409}%
(2014), 219-243.

\bibitem{ye1} S. Ye, \textit{Symmetries of flat manifolds, Jordan property
and the general Zimmer program}, arXiv:1704.03580.

\bibitem{zm} R. Zimmer and D.W. Morris, \textit{Ergodic Theory, Groups, and
Geometry}, American Mathematical Society, Providence, 2008.
\end{thebibliography}
\end{document}